\documentclass[11pt,oneside,a4paper]{amsart}
\usepackage[british]{babel}

\usepackage{graphicx}
\usepackage{amscd}
\usepackage{amsmath}
\usepackage{amsfonts}
\usepackage{amssymb}
\usepackage{mathrsfs}
\usepackage{comment}
\excludecomment{mycomment} 
\usepackage{color}
\usepackage[usenames,dvipsnames,table,xcdraw]{xcolor}
\usepackage{pdfsync}
\usepackage{bbm}
\usepackage{dsfont}
\usepackage[colorlinks=true]{hyperref}
\usepackage{booktabs}
\usepackage{float}
\usepackage{wrapfig}
\usepackage{caption}
\usepackage{subcaption}
\usepackage{longtable}
\usepackage{listings}
\usepackage[T1]{fontenc}
\usepackage[scaled]{beramono}
\usepackage{tikz}
\usepackage{pgffor}
\usepackage[all]{xy}
\usepackage{bbold}
\usepackage{enumerate}
\usepackage{lmodern}

\definecolor{trp}{rgb}{1,1,1}

\definecolor{red}{rgb}{1,0,.2}

\usepackage{amsthm}
\theoremstyle{plain}
\newtheorem{theorem}{Theorem}[section]

\newtheorem{corollary}[theorem]{Corollary}

\newtheorem{definition}[theorem]{Definition}

\newtheorem{lemma}[theorem]{Lemma}

\newtheorem{proposition}[theorem]{Proposition}

\newtheorem{remark}[theorem]{Remark}

\newtheorem*{terminology}{Terminology}

\numberwithin{equation}{section}

\newcommand*{\arabicdec}[1]{\the\numexpr\value{#1}\relax}

\newcommand{\nice}{limit-irreducible}
\newcommand{\lighto}{light overlaps}
\newcommand{\crosso}{cross overlaps}
\newcommand{\ordero}{order of overlapping}

\usepackage{anysize}
\papersize{24.5cm}{16.2006cm}

\marginsize{1cm}{1cm}{0cm}{0.5cm}

\usepackage{caption}

\definecolor{blue}{rgb}{0,0,1}

\definecolor{red}{rgb}{1,0,.7}

\definecolor{myred1}{RGB}{255, 0, 0}

\usepackage{tikz}
\usetikzlibrary{cd,decorations.pathreplacing,positioning,arrows}

\usepackage{xr}

\begin{document}
\title[CPLIFS]{Fractal dimensions of continuous piecewise linear iterated function systems}

\author{R. D\'aniel Prokaj}
\address{ R. D\'aniel Prokaj,  Alfréd Rényi Institute of Mathematics, 
Reáltanoda u. 13-15., 1053 Budapest, Hungary} 
\email{prokajrd@math.bme.hu}

\author{Peter Raith}
\address{Peter Raith, Fakultät für Mathematik, Universität Wien, 
Oskar-Morgenstern-Platz 1, 1090 Wien, Austria}
\email{peter.raith@univie.ac.at}

\author{K\'aroly Simon}
\address{K\'aroly Simon, Department of Stochastics, Institute of Mathematics, Budapest University of Technology and Economics, Műegyetem rkp. 3., 1111 Budapest, Hungary,
and ELKH-BME Stochastics Research Group, P.O. Box 91, 1521 Budapest, Hungary}
\email{simonk@math.bme.hu}

\thanks{
\\ \indent 
The research of the first and third authors was partially supported by National Research, Development and Innovation Office - NKFIH, Project K142169. This work was partially supported by the grant Stiftung Aktion Österich Ungarn 103öu6. 
}  

\begin{abstract}
We consider iterated function systems on the real line that consist of continuous, piecewise linear functions.Under a mild separation condition, we show that the Hausdorff and box dimensions of the attractor are equal to the minimum of 1 and the exponent which comes from the most natural system of covers of the attractor.
\end{abstract}
\date{\today}

\maketitle


\thispagestyle{empty}
\section{Introduction}\label{md20}

Iterated Function Systems (IFS) on the line consist of finitely many strictly contracting self-mappings of $\mathbb{R}$. It was proved by Hutchinson \cite{hutchinson1981fractals}
that
for every IFS $\mathcal{F}=\left\{f_k\right\}_{k=1}^{m}$ there is a unique non-empty compact set $\Lambda$ which is
called the attractor of the IFS $\mathcal{F}$ and defined by
 \begin{equation}\label{cr65}
   \Lambda=\bigcup\limits_{k=1}^{m}f_k(\Lambda).
 \end{equation}

For every IFS $\mathcal{F}$ there exists a unique ``smallest'' non-empty compact interval $I$  which is sent into itself by all the mappings of $\mathcal{F}$:
\begin{equation}\label{cr22}
  I:=\bigcap\left\{ J \big\vert\:
  J\subset \mathbb{R}\mbox{ compact interval} :
  f_k(J)\subset J, \forall k\in[m]
  \right\},
\end{equation}
where $[m]:=\left\{ 1,\dots  ,m \right\}$. 
It is easy to see that 
\begin{equation}
\label{cr15}
\Lambda=\bigcap\limits_{n=1}^{\infty}
\bigcup\limits_{(i_1,\dots  ,i_n)\in[m]^n}
I_{i_1\dots  i_n} ,
\end{equation}
where $I_{i_1\dots  i_n}:=
f_{i_1\dots  i_n}(I)$ are the \textbf{cylinder intervals}, and we use the common shorthand notation $f_{i_1\dots  i_n}:=f_{i_1}\circ\cdots \circ f_{i_n}$ for an
$(i_1,\dots  ,i_n)\in [m]^n$. 

The IFSs we consider in this paper are consisting of piecewise linear functions. Thus their derivatives might change at some points, but they are linear over given intervals of $\mathbb{R}$. We always assume that the functions are continuous, piecewise linear, strongly contracting with non-zero slopes, and that the slopes can only change at finitely many points.

Let $\mathcal{F}=\left\{f_k\right\}_{k=1}^{m}$ be a CPLIFS and 
$I \subset \mathbb{R}$ be the compact interval defined in \eqref{cr22}. For any $k\in [m]$
let $l(k)$ be the number of breaking points $\left\{b_{k,i}\right\}_{i=1}^{l(k)}$ of $f_k$.
They determine the $l(k)+1$ open intervals of linearity
$\left\{J_{k,i}\right\}_{i=1}^{l(k)+1}$. We write $S_{k,i}$ for the contracting similarity on $\mathbb{R}$ that satisfies
$S_{k,i}|_{J_{k,i}}\equiv f_k|_{J_{k,i}}$.
We define $\left\{\rho_{k,i}\right\}_{k\in[m],i\in [l(k)+1]}$ and $\left\{t_{k,i}\right\}_{k\in[m],i\in [l(k)+1]}$  such that
\begin{equation}\label{cr56}
  S_{k,i}(x)=\rho_{k,i}x+t_{k,i}.
\end{equation}
We say that $\mathcal{S}_{\mathcal{F}}:=\left\{S_{k,i}\right\}_{k\in[m],i\in [l(k)+1]}$ is the \textbf{self-similar IFS  generated by the CPLIFS  $\mathcal{F}$}.

We introduce the \textbf{natural pressure function}  
\begin{equation}\label{cr64}
    \Phi(s):=\limsup_{n\rightarrow\infty}\frac{1}{n}\log \sum_{i_1\dots i_n} |I_{i_1\dots i_n}|^s.
\end{equation}
In \cite{barreira1996non}, Barreira showed that $\Phi(s): \mathbb{R}_+\to \mathbb{R}$ is a strictly decreasing function with $\Phi(0)>0$ and $\lim_{s\to\infty} \Phi(s)=-\infty$. Hence we can define the \textbf{natural dimension} of $\mathcal{F}$ as  
\begin{equation}\label{cr61}
    s_{\mathcal{F}}:=(\Phi)^{-1}(0).
\end{equation}
We note that he called $\Phi(s)$ the non-additive upper capacity topological pressure. He also proved that the upper box dimension is always bounded from above by the natural dimension.
Moreover, in the very special case when there are no breaking points, that is $\mathcal{F}
=\left\{ f_k(x)=\rho _kx+t_k \right\}_{k=1}^{m}$ is self-similar, $s_{\mathcal{F}}$ is the so-called similarity dimension that is $\sum_{k=1  }^{m }\rho _{k}^{ s_{\mathcal{F}}}=1$.
\begin{corollary}[Barreira]\label{cr12}
  For any IFS $\mathcal{F}$ on the line 
  \begin{equation}\label{cr13}
    \overline{\dim}_{\rm B}  \Lambda\leq s_{\mathcal{F}}.
  \end{equation}
\end{corollary}
Here $\dim_{\rm B}\Lambda$ stands for the box dimension of the attractor. For the definition of fractal dimensions we refer the reader to \cite{falconer1997techniques}.
The inequality \eqref{cr13} also follows from \cite[Theorem~8.8]{falconer1986geometry}, but to the best of our knowledge, in this explicit form it was first stated in \cite{barreira1996non}. 
It follows that the Hausdorff dimension of the attractor $\dim_{\rm H}\Lambda$ is also bounded from above by the natural dimension $s_{\mathcal{F}}$. We will show that in a sense typically, these dimensions are actually equal.

A continuous piecewise linear iterated function system $\mathcal{F}=\{f_k\}_{k=1}^m$ is uniquely determined by the slopes $\{\rho_{k,1},\dots,\rho_{k,l(k)+1}\}_{k=1}^m$, the breaking points $\{b_{k,1},\dots,$ $b_{k,l(k)}\}_{k=1}^m$ and the vertical translations $\{f_k(0)\}_{k=1}^m$ of its functions. The latter two are called the translation parameters of $\mathcal{F}$.

\begin{terminology}\label{z99} 
    Given a property which is meaningful for all CPLIFSs.
     We say that this property is $\pmb{\dim_{\rm P}}$\textbf{-typical} if the set of translation parameters for which it does not hold has less than full packing dimension, for any fixed vector of slopes.
\end{terminology}
\begin{theorem}\label{md10}
We write $\Lambda _{\mathcal{F}}$ 
for the attractor of a CPLIFS  $\mathcal{F}$.
Then the following property is $\dim_{\rm P}  $-typical:
  \begin{equation}\label{md09}
    \dim_{\rm H}\Lambda_{\mathcal{F}} = \dim_{\rm B}\Lambda_{\mathcal{F}} = \min \{1, s_{\mathcal{F}}\}.
  \end{equation}
\end{theorem}

To prove our main result we need a separation condition that was introduced by M. Hochman \cite{hochman2014self} for self-similar iterated function systems.
Let $g_1(x)=\rho_1x+\tau_1$ and $g_2(x)=\rho_2x+\tau_2$ be two similarities on $\mathbb{R}$ with $\rho_1,\rho_2\in\mathbb{R}\setminus \{0\}$ and $\tau_1,\tau_2\in\mathbb{R}$.
We define the distance of these two functions as 
\begin{equation}\label{md17}
  \mathrm{dist}(g_1,g_2):=\begin{cases}
    \vert \tau_1-\tau_2\vert \mbox{, if } \rho_1=\rho_2; \\
    \infty \mbox{, otherwise.}
  \end{cases}
\end{equation}
\begin{definition}\label{md16}
  Let $\mathcal{F}=\{f_k(x)\}_{k=1}^m$ be a self-similar IFS on $\mathbb{R}$. We say that $\mathcal{F}$ satisfies the \textbf{Exponential Separation Condition (ESC)} if there exists a $c>0$ and a strictly increasing sequence of natural numbers $\{n_l\}_{l=1}^{\infty}$ such that 
  \begin{equation}\label{md15}
    \mathrm{dist}(f_{\mathbf{i}},f_{\mathbf{j}})\geq c^{n_l}\mbox{, for all } l>0 \mbox{ and for all } \mathbf{i},\mathbf{j}\in [m]^{n_l}, \mathbf{i}\neq\mathbf{j}. 
  \end{equation}
\end{definition}
Hochman proved that the ESC is a $\dim_{\rm P}$-typical property of self-similar IFS \cite[Theorem~1.10]{Hochman_2015}. Prokaj and Simon extended this result by showing that it is a $\dim_{\rm P}$-typical property of a CPLIFS that the generated self-similar system satisfies the ESC \cite[Fact~4.1]{prokaj2021piecewise}.
 These two results together with the next theorem yields Theorem \ref{md10}.

\begin{theorem}\label{md46}
  Let $\mathcal{F}$ be a CPLIFS with generated self-similar system $\mathcal{S}$ and attractor $\Lambda$. If $\mathcal{S}$ satisfies the ESC, then
  \begin{equation}\label{md45}
    \dim_{\rm H}\Lambda = \dim_{\rm B}\Lambda = \min \{1, s_{\mathcal{F}}\}.
  \end{equation} 
\end{theorem}
We are going to prove this theorem with the help of Markov diagrams.

\subsection{Self-similar IFSs on the line}\label{x99}
The first breakthrough result of this field in the last decade was due to Hochman \cite{hochman2014self}. He proved that ESC implies that the Hausdorff dimension of a self-similar measure (see \cite[p.37]{falconer1997techniques}) is equal to the minimum of $1$ and the ratio obtained by dividing the entropy by the Lyapunov exponent of this measure (for the definitions see \cite[p. 3]{jordan2020dimension}). 
As an immediate application of this result, we get that ESC implies \eqref{md45} for self-similar IFS on the line.
 
Shmerkin \cite{shmerkin2019furstenberg} 
proved an analogous result for the $L^q$ dimension of self-similar 
measures in 2019. 
Using similar ideas, Jordan and Rapaport \cite{jordan2020dimension} extended 
Hochman's result mentioned above from self-similar measures to the projections of ergodic measures. \newline 
Using this result, Prokaj and Simon \cite[Corollary 7.2]{prokaj2021piecewise} proved that ESC also implies that formula \eqref{md45} holds for graph-directed self-similar attractors on the line. This was an essential tool of the proof of Theorem \ref{md14} below.

\subsection{Earlier results about CPLIFS}\label{md23}

Let $\mathcal{F}=\{f_k(x)\}_{k=1}^m$ be a CPLIFS on $\mathbb{R}$ with attractor $\Lambda$.
If $\Lambda$ does not contain any breaking point, then we call $\mathcal{F}$ \textbf{regular}. Prokaj and Simon showed in \cite{prokaj2021piecewise} that for any regular CPLIFS there is a self-similar graph-directed IFS with the same attractor. This observation led to the following theorem \cite[Theorem~2.2]{prokaj2021piecewise}
\begin{theorem}\label{md14}
  Let $\mathcal{F}$ be a regular CPLIFS for which the generated self-similar IFS satisfies the ESC. Then
  \begin{equation}\label{md12}
    \dim_{\rm H}\Lambda = \dim_{\rm B}\Lambda = \min\{1,s_{\mathcal{F}}\}.
  \end{equation}  
\end{theorem}

For a $k\in[m]$, let $\rho_k:=\max_{j\in[l(k)+1]}\vert\rho_{k,j}\vert$ be the biggest slope of $f_k$ in absolute value. 
We say that $\mathcal{F}$ is \textbf{small} if the following two assertions hold:
\begin{enumerate}[{\bf (a)}]
  \item $\sum_{k=1}^m \rho_k<1$.
  \item The second requirement depends on the injectivity of the functions of the system:
        \begin{enumerate}[{\bf (i)}]
          \item If $f_k$ is injective, then $\rho_k <\frac{1}{2}$;
          \item If $f_k$ is not injective, then  $\rho_k <\frac{1-\max_j\rho_j}{2}$.
        \end{enumerate} 
\end{enumerate} 

\begin{proposition}[{\cite[Proposition~2.3]{prokaj2021piecewise}}]\label{md13}
  For small CPLIFS, regularity is a $\dim_{\rm P}$-typical property.
\end{proposition}

The dimension theory of some atypical CPLIFS families is discussed in \cite{prokaj2022special}.
By combining Proposition \ref{md13}, Theorem \ref{md14} and the result of M. Hochman \cite[Theorem~1.10]{Hochman_2015}, Prokaj and Simon also showed that the equalities \eqref{md12} are $\dim_{\rm P}$-typical properties of small CPLIFSs. In this paper, we will extend this result by showing that the same holds without restrictions on the slopes.

\section{Markov Diagrams}\label{md22}

P. Raith and F. Hofbauer proved results on the dimension of expanding piecewise monotonic systems using the notion of Markov diagrams \cite{raith1989hausdorff,raith1994continuity,hofbauer1986piecewise}. 
We will define the Markov diagram in a similar fashion for CPLIFSs, and then use it to prove that $s_{\mathcal{F}}$ equals the Hausdorff dimension of the attractor for non-regular systems as well, under some weak assumptions.

\subsection{Building Markov diagrams}\label{md21}

Let $\mathcal{F}=\{f_k\}_{k=1}^m$ be a CPLIFS, and let $I$ be the interval defined by \eqref{cr22}. Writing $I_k:=f_k(I)$ and $\mathcal{I}=\cup_{k=1}^m I_k$, we define the \textbf{expanding multi-valued mapping associated to} $\mathcal{F}$ as 
\begin{equation}\label{md59}
    T:\mathcal{I}\mapsto \mathcal{P}(\mathcal{P}(I)),\quad 
    T(y):= \big\{\{x\in I: f_k(x)=y\}\big\}_{k=1}^m .
\end{equation}
That is the image of any Borel subset $A\subset\mathcal{I}$ is 
\[
    T(A)= \big\{\{x\in I: f_k(x)\in A\}\big\}_{k=1}^m.
\]
For $k\in[m], j\in[l(k)+1]$, we define $f_{k,j}:J_{k,j}\to I_k$ as the unique linear function that satisfies $\forall x\in J_{k,j}: f_k(x)=f_{k,j}(x)$.
We call the expansive linear functions 
\begin{align}\label{md58}
    \forall k\in[m], \forall j\in &[l(k)+1]:\quad f_{k,j}^{-1}: f_k(J_{k,j}) \to J_{k,j}, \\
    &\forall x\in J_{k,j}: f_{k,j}^{-1}(f_k(x))=x \nonumber
\end{align}
the \textbf{branches} of the multi-valued mapping $T$. As the notation suggests, these are the local inverses of the elements of $\mathcal{F}$.

\begin{figure}[t]
  \centering
  \includegraphics[width=12cm]{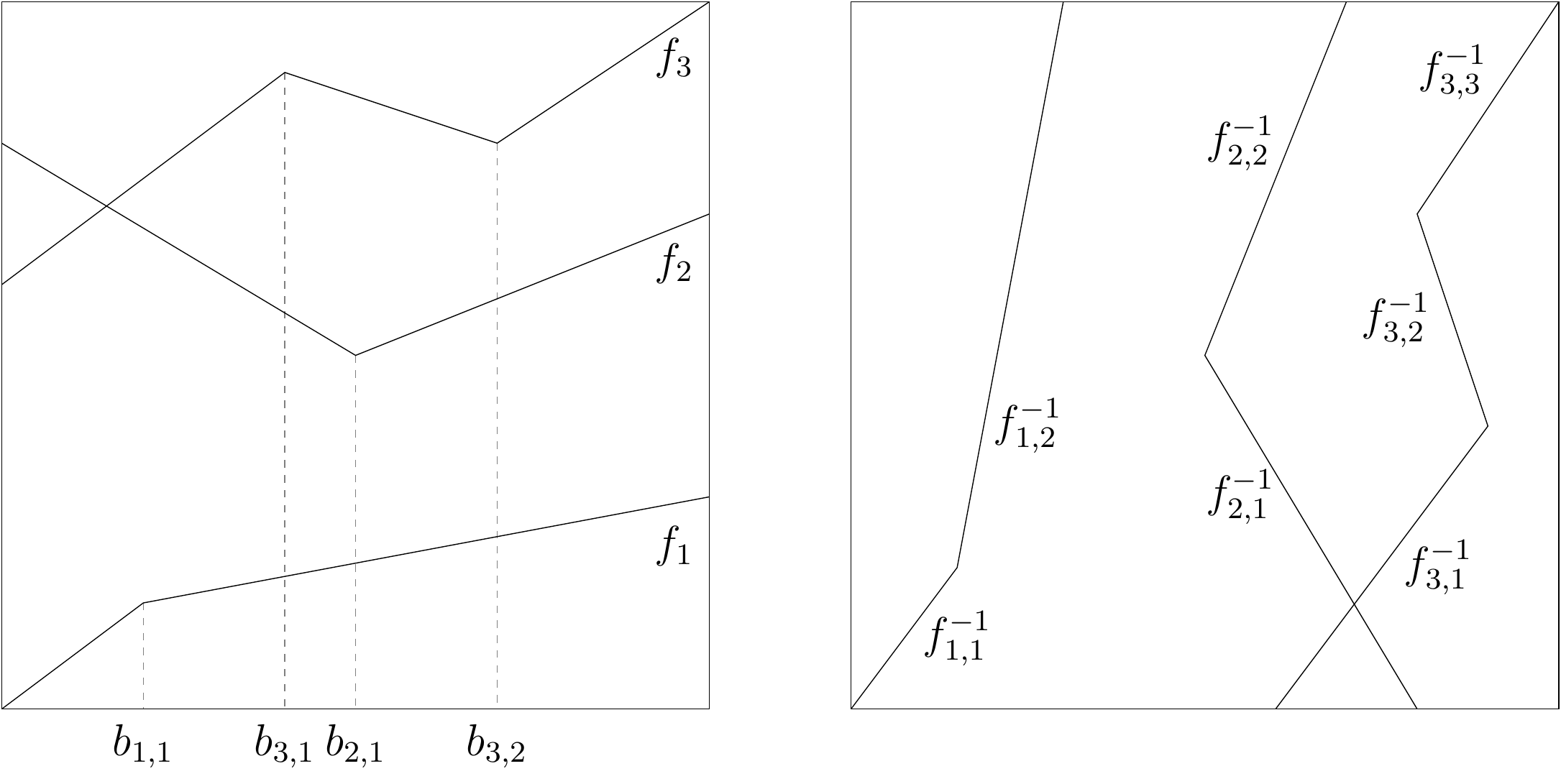}
  \caption{A CPLIFS $\mathcal{F}=\{f_k\}_{k=1}^m$ is on the left with its associated expansive multi-valued mapping $T$ on the right. The critical points are colored with blue.}\label{md25}
\end{figure}

\begin{definition}\label{md57}
    We define the \textbf{set of critical points} as 
      \begin{gather*}
        \mathcal{K}:=\cup_{k=1}^m \{ f_k(0),f_k(1)\} \bigcup 
        \cup_{k=1}^m \cup_{j=1}^{l(k)} f_k(b_{k,j}) \bigcup \\
        \left\{ x\in \mathcal{I} \big\vert \exists k_1,k_2\in [m], 
        \exists j_1\in[l(k_1)], \exists j_2\in[l(k_2)]:  
        f_{k_1, j_1}^{-1}(x)=f_{k_2, j_2}^{-1}(x) \right\}.
      \end{gather*}
\end{definition}

\begin{definition}\label{md56}
    We call the partition of $\mathcal{I}$ into closed intervals defined by the set of critical points $\mathcal{K}$ the \textbf{monotonicity partition} $\mathcal{Z}_0$ of $\mathcal{F}$. We call its elements \textbf{monotonicity intervals}.
\end{definition}

\begin{definition}\label{md55}
    Let $Z\in \mathcal{Z}_0$. We say that $D$ is a \textbf{successor} of $Z$ and we write $Z\to D$ if
    \begin{equation}\label{md54}
      \exists Z_0\in \mathcal{Z}_0, Z^{\prime}\in T(Z): D=Z_0\cap Z^{\prime}.
    \end{equation}
    Further, we write $Z\to_{k,j} D$ if 
    \[
      \exists Z_0\in \mathcal{Z}_0: D=Z_0\cap f_{k,j}^{-1}(Z).  
    \]
    The set of successors of $Z$ is denoted by $w(Z):=\{D\vert Z\to D\}$.
\end{definition}

Similarly, we define $w(\mathcal{Z}_0)$ as the set of the successors of all elements of $\mathcal{Z}_0$. That is 
\begin{equation}\label{md96}
  w(\mathcal{Z}_0):=\cup_{Z\in \mathcal{Z}_0} w(Z).
\end{equation}

\begin{definition}\label{md95}
  We say that $(\mathcal{D},\to)$ is the \textbf{Markov Diagram of} $\mathcal{F}$ \textbf{ with respect to} $\mathcal{Z}_0$ if $\mathcal{D}$ is the smallest set containing $\mathcal{Z}_0$ such that $\mathcal{D}=w(\mathcal{D})$. For short, we often call it the Markov diagram of $\mathcal{F}$. 
\end{definition}

\begin{remark}\label{md36}
    We can similarly define the Markov diagram of $\mathcal{F}$ with respect to any finite partition $\mathcal{Z}_0^{\prime}$ of $\mathcal{I}$. 
\end{remark}

We often use the notation 

\begin{equation}\label{md60}
  \mathcal{D}_n:=\cup_{i=0}^n w^i(\mathcal{Z}_0) \mbox{, where }
  w^i(\mathcal{Z}_0)=\underbrace{w\circ\cdots\circ w}_{i \mbox{ times}} (\mathcal{Z}_0).
\end{equation}
Obviously, 
\begin{equation}\label{md93}
  \mathcal{D}=\cup_{i\geq 0} w^i(\mathcal{Z}_0).
\end{equation}
If the union in \eqref{md93} is finite, we say that \textbf{the Markov diagram is finite}. We define recursively the $\pmb{n}$\textbf{-th level of the Markov diagram} as
\begin{equation}\label{md53}
   \mathcal{Z}_n:=\omega(\mathcal{Z}_{n-1})\setminus \cup_{i=0}^{n-1} \mathcal{Z}_{i},
\end{equation}
for $n\geq 1$.

One can imagine the Markov diagram as a (potentially infinitely big) directed graph, with vertex set $\mathcal{D}$. Between $C,D\in \mathcal{D}$, we have a directed edge $C\to D$ if and only if $D\in w(C)$. We call the Markov diagram \textbf{irreducible} if there exists a directed path between any two intervals $C,D\in \mathcal{D}$. 
In the next lemma we prove that by choosing an appropriate refinement $\mathcal{Y}_0$ of $\mathcal{Z}_0$, the Markov diagram $(\mathcal{D}^{'},\to)$ of $\mathcal{F}$ with respect to $\mathcal{Y}_0$ always has an irreducible subdiagram. Further, the elements of $\mathcal{D}^{\prime}$ cover $\Lambda$. It implies that $(\mathcal{D}^{'},\to)$ is sufficient to describe the orbits of the points of $\Lambda$. That is, Lemma \ref{md91} enables us to assume that the Markov diagram $(\mathcal{D},\to)$ of $\mathcal{F}$ is irreducible without loss of generality. 

\begin{lemma}\label{md91}
  Let $\mathcal{F}$ be a CPLIFS with attractor $\Lambda$, and let $(\mathcal{D}(\mathcal{Y}_0),\to)$ be its Markov diagram with respect to some finite refinement $\mathcal{Y}_0$ of the monotonicity partition $\mathcal{Z}_0$. For the right choice of $\mathcal{Y}_0$, 
  there exists an irreducible subdiagram $(\mathcal{D}^{'},\to)$ of $(\mathcal{D}(\mathcal{Y}_0),\to)$, such that the elements of $\mathcal{D}^{\prime}$ cover $\Lambda$.
\end{lemma}

\begin{proof}
    For $k\in[m]$, let $\phi_k$ be the fixed point of $f_k$. We assume without loss of generality that $\phi_i\leq \phi_j$ if $i<j$ for $i,j\in[m]$. Let $\mathcal{Y}_0$ be the refinement of $\mathcal{Z}_0$ with $\phi_1$. 
    There are at most two intervals in $\mathcal{Y}_0$ that ends in $\phi_1$, we write $Y_1$ and $Y_2$ for them. Let $\mathcal{D}^{'}$ be the set that contains $Y_1, Y_2$ and all of their successors which intersect $\Lambda$. In particular,
    \begin{equation}\label{md08}
        \mathcal{D}^{\prime}:= \{Y_1,Y_2\} \bigcup \{C\subset \mathcal{I} \vert \exists n>0, \exists j\in\{1,2\} : C\in w^n(Y_i) \:\&\: C\cap \Lambda\neq\emptyset\}. 
    \end{equation} 
    Obviously, $(\mathcal{D}^{'},\to)$ is a subdiagram of $(\mathcal{D}(\mathcal{Y}_0),\to)$. Let $C\in \mathcal{D}^{\prime}$ be an arbitrary interval. The attractor $\Lambda$ is the invariant set of $\mathcal{F}$, thus some element of $w(C)$ also intersects $\Lambda$. It follows that there are no deadends in $(\mathcal{D}^{\prime},\to)$, every directed path can be continued within the subdiagram.

    Since $C\cap\Lambda\neq\emptyset$ and the elements of $\mathcal{F}$ are strict contractions, there exist $N>0$ and $\mathbf{k}=(k_1,\dots ,k_N)\in \{1,\dots ,m\}^N$ such that $f_{\mathbf{k}}(I)\cap C\neq\emptyset$ and $\rho_{\min}^{-N} \vert C\vert > \vert I\vert$, where $\rho_{\min}<1$ is the smallest slope of the system $\mathcal{F}$. Then, the elements of the set 
    \begin{equation}\label{md06}
      \left\{ C_N\in \mathcal{D}^{\prime}\vert 
      \exists j_1,\dots ,j_N: 
      C\to_{k_1,j_1}\cdots \to_{k_N,j_N} C_N 
      \right\}
    \end{equation}
    cover $I$ and hence the attractor $\Lambda$. We just obtained that 
    \begin{equation}\label{md07}
      \forall C\in \mathcal{D}^{\prime}, \forall U\subset \mathcal{I} : \exists n>0, \exists C^{\prime}\in w^n(C) 
      \mbox{ such that } C^{\prime}\cap U\neq \emptyset.
    \end{equation} 
    
    By applying \eqref{md07} to a $U$ neighbourhood of $\phi_1$, it follows that from
    every $C\in \mathcal{D}^{\prime}$ there is a directed path in $(\mathcal{D}^{'},\to)$ to some $C^{\prime}\in \mathcal{D}^{\prime}$ that ends in $\phi_1$.
    
    Since $\phi_1$ is the fixed point of the strict contraction $f_1$, if $\phi_1\in C\in \mathcal{D}^{\prime}$, then there must be a directed path in $(\mathcal{D}^{'},\to)$ from $C$ to either $Y_1$ or $Y_2$. By applying \eqref{md07} again, there are directed paths in $(\mathcal{D}^{'},\to)$ between $Y_1$ and $Y_2$, thus it is indeed an irreducible subdiagram.
\end{proof}

\subsection{Connection to the natural pressure}\label{md19}

Similarly to graph-directed iterated function systems, we associate a matrix to Markov diagrams, that will help us determine the Hausdorff dimension of the corresponding CPLIFS (see \cite{falconer1997techniques}).
\begin{definition}\label{md90}
  Let $\mathcal{F}=\{f_k\}_{k=1}^m$ be a CPLIFS, and write $(\mathcal{D},\to )$ for its Markov diagram. We define the matrix $\mathbf{F}(s):=\mathbf{F}_{\mathcal{D}}(s)$ indexed by the elements of $\mathcal{D}$ as
  \begin{equation}\label{md89}
    \left[ \mathbf{F}(s)\right]_{C,D} := \begin{cases}
      \sum_{(k,j): C\to_{(k,j)} D} \vert f_{k,j}^{'}\big\vert^s \mbox{, if $C\to D$} \\
      0 \mbox{, otherwise.}
    \end{cases}  
  \end{equation}
  We call $\mathbf{F}_{\mathcal{D}}(s)$ the \textbf{matrix associated to the Markov diagram} $(\mathcal{D},\to)$.
\end{definition}

We used in the definition, that for a $D\in \mathcal{D}$ with $C\to_{(k,j)} D$ the derivative of $f_{k,j}$ over $D$ is a constant number. That is each element of $\mathbf{F}(s)$ is either zero or a sum of the $s$-th power of some contraction ratios.

This matrix can be defined for any $\mathcal{C}\subset \mathcal{D}$ as well, by choosing the indices from $\mathcal{C}$ only. We write $\mathbf{F}_{\mathcal{C}}(s)$ for such a matrix. It follows that $\mathbf{F}_{\mathcal{C}}(s)$ is always a submatrix of $\mathbf{F}(s)$ for $\mathcal{C}\subset \mathcal{D}$. 

We write $\mathcal{E}_{\mathcal{C}}(n)$ for the set of $n$-length directed paths in the graph $(\mathcal{C},\to)$.
\begin{equation}\label{md87}
  \begin{aligned}
    \mathcal{E}_{\mathcal{C}}(n):= &\{ ((k_1,j_1),\dots ,(k_n,j_n)) \big\vert 
      \exists C_1,\dots ,C_{n+1}\in \mathcal{C}: \\
      &\forall q\in [n] \exists k_q\in[m], j_q\in[l(k)]) : 
      C_{q}\to_{(k_q,j_q)} C_{q+1} \}.
  \end{aligned}
\end{equation}

An $n$-length directed path here means $n$ many consecutive directed edges, and we identify each such path with the labels of the included edges in order.
Each path in $(\mathcal{D},\to)$ of infinite length represents a point in $\Lambda$, and each point is represented by at least one path. 
Similarly, for $\mathcal{C}\subset \mathcal{D}$ the points defined by the natural projection of infinite paths in $(\mathcal{C},\to)$ form an invariant set $\Lambda_{\mathcal{C}}\subset \Lambda$.
We define the natural pressure of these sets as
\begin{equation}\label{md88}
  \Phi_{\mathcal{C}}(s):= \limsup_{n\to \infty}\frac{1}{n}\log
  \sum_{\mathbf{k}} \vert I_{\mathbf{k}}\vert^s,
\end{equation} 
where the sum is taken over all $\mathbf{k}=(k_1,\dots k_n)$ for which $\exists j_1,\dots j_n: ((k_1,j_1),$ $\dots ,(k_n,j_n))\in \mathcal{E}_{\mathcal{C}}(n)$, and
$I$ is the interval defined in \eqref{cr65}.
By the definition of $\mathcal{D}$ it is easy to see that $\Phi_{\mathcal{D}}(s)=\Phi (s)$. 

\begin{remark}\label{md35}
    Let $\mathcal{Y}_0$ be a finite refinement of the monotonicity partition $\mathcal{Z}_0$, and let $(\mathcal{D}^{\prime},\to)$ be the Markov diagram of $\mathcal{F}$ with respect to $\mathcal{Y}_0$. Obviously, 
    \begin{equation}
        \forall s\geq 0: \Phi_{\mathcal{D}}(s)=\Phi_{\mathcal{D}^{\prime}}(s).
    \end{equation}
\end{remark}

We will show, that the unique zero of the function $\Phi_{\mathcal{D}}(s)$ can be approximated by the root of $\Phi_{\mathcal{C}}(s)$ for some $\mathcal{C}\subset \mathcal{D}$. To show this, we need to connect the function $\Phi_{\mathcal{C}}(s)$ to the matrix $\mathbf{F}_{\mathcal{C}}(s)$.

As an operator, $(\mathbf{F}_{\mathcal{D}}(s))^n$ is always bounded in the $l^{\infty}$-norm. Thus we can define 
\[
\varrho (\mathbf{F}_{\mathcal{C}}(s)) := \lim_{n\to \infty} \lVert (\mathbf{F}_{\mathcal{C}}(s))^n\rVert_{\infty} ^{1/n}. 
\] 

\begin{lemma}\label{md84}
  Let $\mathcal{C}\subset \mathcal{D}$. If $(\mathcal{C},\to)$ is irreducible, then 
  \begin{equation}\label{md86}
    \Phi_{\mathcal{C}}(s)\leq \log \varrho (\mathbf{F}_{\mathcal{C}}(s)).
  \end{equation}
  If $(\mathcal{C},\to)$ is irreducible and finite, then
  \begin{equation}\label{md85}
    \Phi_{\mathcal{C}}(s)= \log \varrho (\mathbf{F}_{\mathcal{C}}(s)).
  \end{equation}
\end{lemma}

\begin{proof}
  First only assume that $(\mathcal{C},\to)$ is irreducible. Since it is irreducible, we can think about $(\mathcal{C},\to)$ as the Markov diagram of some IFS with level $n$ cylinder intervals $\{ I_{\mathbf{i}}\}_{\mathbf{i}\in\mathcal{E}_{\mathcal{C}}(n)}$. 

  Fix $\mathbf{k}=(k_1,\dots ,k_n)\in[m]^n$. There are at least one, but possibly several directed paths of length $n$ in the graph with labels $((k_1,j_1),\dots ,(k_n,j_n))$ for some $j_1,\dots j_n$. 
  Each of these paths correspond to a unique entry in $\mathbf{F}^n_{\mathcal{C}}(s)$. 
  The biggest one of these entries times $|I|$ is an upper bound on $|I_{\mathbf{k}}|^s$.
  Since every $n$ length path starts at some element of $\mathcal{Z}_0$, we obtain that 
  \begin{equation}\label{md52}
    \sum_{\mathbf{k}} \vert I_{\mathbf{k}}\vert^s \leq |\mathcal{Z}_0|\cdot \Vert\mathbf{F}^n_{\mathcal{C}}(s)\Vert_{\infty}\cdot |I|,
  \end{equation}
  where the sum is taken over all $\mathbf{k}=(k_1,\dots k_n)$ for which $\exists j_1,\dots j_n: ((k_1,j_1),$ $\dots ,(k_n,j_n))\in \mathcal{E}_{\mathcal{C}}(n)$.
  By taking logarithm on both sides, dividing them by $n$, and then taking the limit as $n\to\infty$, we obtain \eqref{md86}.
  
  Now assume that $\mathcal{C}$ is finite, and write $N$ for the highest level of the Markov diagram. It means that for every $|\mathbf{k}|\geq N$, the cylinder interval $I_{\mathbf{k}}$ is contained in an element of $\mathcal{C}$. That is 
  \begin{equation}\label{md51}
    \forall n\geq 0: \min_{|\mathbf{l}|=N} |I_{\mathbf{l}}| \Vert\mathbf{F}^n_{\mathcal{C}}(s)\Vert_{\infty} \leq \sum_{\mathbf{k}} \vert I_{\mathbf{k}}\vert^s \leq 
    \max_{|\mathbf{l}|=N} |I_{\mathbf{l}}| |\mathcal{Z}_0| \Vert\mathbf{F}^n_{\mathcal{C}}(s)\Vert_{\infty},
  \end{equation} 
  where the sum in the middle is taken over all $\mathbf{k}=(k_1,\dots k_{N+n})$ for which $$\exists j_1,\dots j_{N+n}: ((k_1,j_1),\dots ,(k_{N+n},j_{N+n}))\in \mathcal{E}_{\mathcal{C}}(N+n).$$
  It follows that \eqref{md85} holds. 
\end{proof}

Let $(\mathcal{C}_1,\to),(\mathcal{C}_2,\to),\dots$ be an increasing sequence of irreducible subgraphs of $(\mathcal{D},\to)$.
It follows from Seneta's results \cite[Theorem~1]{seneta2006non} that the so called $R$-values of the matrices $\mathbf{F}_{\mathcal{C}_n}(s)$ converge to the $R$-value of $\mathbf{F}(s)$. 
For an irreducible finite matrix $\mathbf{A}$ we always have $R(\mathbf{A})=\frac{1}{\varrho(\mathbf{A})}$, then the convergence of the spectral radius $\varrho(\mathbf{F}_{\mathcal{C}_n}(s))$ to $\varrho(\mathbf{F}(s))$ follows. 
Altough $\mathbf{F}(s)$ may not be finite, the relation $R(\mathbf{F}(s))=\frac{1}{\varrho(\mathbf{F}(s))}$ can still be guaranteed by some assumptions. That is why the following property has a crucial role in our proofs.

\begin{definition}\label{md82}
  Let $\mathcal{F}$ be a CPLIFS and $\mathcal{Y}$ be a finite refinement of the monotonicity partition $\mathcal{Z}_0$. Let $(\mathcal{D}(\mathcal{Y}),\to)$ be the Markov diagram of $\mathcal{F}$ with respect to $\mathcal{Y}$, and let $\mathbf{F}(\mathcal{Y},s)$ be its associated matrix.

  We say that the CPLIFS $\mathcal{F}$ is \textbf{\nice{}} if there exists a $\mathcal{Y}$ such that for all $s\in(0,\dim_{\rm H}\Lambda]$ the matrix $\mathbf{F}(\mathcal{Y},s)$ has right and left eigenvectors with nonnegative entries for the eigenvalue $\varrho (\mathbf{F}(\mathcal{Y},s))$.

  We call this finite partition $\mathcal{Y}$ a \textbf{\nice{} partition of} $\mathcal{F}$ and \linebreak $(\mathcal{D}(\mathcal{Y}),\to)$ a \textbf{\nice{} Markov diagram of} $\mathcal{F}$. 
\end{definition}

In the next section we show how being \nice{} implies that the Hausdorff dimension of the attractor is equal to the minimum of the natural dimension and $1$. Later, in Section \ref{md24}, we investigate what makes a CPLIFS \nice{}.

\subsection{Proof using the diagrams}\label{md18}

We have already shown a connection between the Markov diagram and the natural pressure of a given CPLIFS. 
Now using this connection, we show that the natural dimension of a \nice{} CPLIFS is  always a lower bound for the Hausdorff dimension of its attractor, by approximating the spectral radius of the Markov diagram with its submatrices' spectral radius.

As in \cite{raith1989hausdorff}, the following proposition holds. 
\begin{proposition}\label{md81}
  Let $\mathcal{F}$ be a \nice{} CPLIFS, and let $(\mathcal{D},\to)$ be its \nice{} Markov diagram.  
  For any $\varepsilon >0$ there exists a $\mathcal{C}\subset \mathcal{D}$ finite subset such that 
  \begin{equation}\label{md80}
    \varrho (\mathbf{F}(s))-\varepsilon \leq \varrho (\mathbf{F}_{\mathcal{C}}(s)) \leq \varrho (\mathbf{F}(s)),
  \end{equation}
  where $\mathbf{F}(s)$ is the matrix associated to $(\mathcal{D},\to)$.
\end{proposition}
The proof is essentially the same as the proof of \cite[Lemma~6~(ii)]{raith1989hausdorff}.
We obtain the following theorem as the combinations of \cite[Theorem~2]{raith1989hausdorff} and \cite[Corollary~7.2]{prokaj2021piecewise}.

\begin{theorem}\label{md72}
  Let $\mathcal{F}$ be a \nice{} CPLIFS with attractor $\Lambda$ and 
  \nice{} Markov diagram $(\mathcal{D},\to)$. Assume that the generated self-similar system of $\mathcal{F}$ satisfies the ESC. Then
  \begin{equation}\label{md71}
    \dim_H \Lambda = \min \{1,s_{\mathcal{F}}\},
  \end{equation}
  where $s_{\mathcal{F}}$ denotes the unique zero of the natural pressure function $\Phi(s)$.
\end{theorem}

The proof is similar to the proof of Theorem 2 in \cite{raith1989hausdorff}. 
\begin{proof}
  By Corollary \ref{cr12}, $\dim_H \Lambda\leq \min\{s_{\mathcal{F}},1\}$ always holds. It is only left to prove the lower bound.

  Choose an arbitrary $t\in (0,s_{\mathcal{F}})$. 
  The natural pressure function is strictly decreasing and has a unique zero at $s_{\mathcal{F}}$, hence $\Phi (t)>0$. The same can be told about the spectral radius of $\mathcal{D}$, according to Remark \ref{md35}.
  $(\mathcal{D},\to)$ is irreducible, but not necessarily finite, thus Lemma \ref{md84} gives 
  \begin{equation}\label{md32}
    0<\Phi (t)\leq\log\varrho \left(\mathbf{F}(t)\right).
  \end{equation} 
  
  According to Proposition \ref{md81}, 
  \begin{equation}\label{md70}
    \exists\; \mathcal{C}\subset \mathcal{D} \mbox{ finite}:\:  
    \log \varrho (\mathbf{F}_{\mathcal{C}}(t))>0.  
  \end{equation}
  Then applying Lemma \ref{md84} again gives 
  \begin{equation}\label{md69}
    0<\log\varrho (\mathbf{F}_{\mathcal{C}}(t))=\Phi_{\mathcal{C}}(t),
  \end{equation}
  since $\mathcal{C}$ is finite. 
  
  For a finite $\mathcal{C}$ the induced attractor $\Lambda_{\mathcal{C}}$ is graph-directed with graph $(\mathcal{C},\to)$. 
  Since the generated self-similar system satisfies the ESC, we already know from \cite[Corollary~7.2]{prokaj2021piecewise} that
  \begin{equation}\label{md68}
    \dim_H \Lambda_{\mathcal{C}}=\min\{s_{\mathcal{C}},1\},
  \end{equation}
  where $s_{\mathcal{C}}$ is the unique root of $\Phi_{\mathcal{C}}(s)$. 
  
  Assume first that $s_{\mathcal{F}}\leq 1$, which implies $s_{\mathcal{C}}\leq 1$ for all $\mathcal{C}\subset \mathcal{D}$.
  Together \eqref{md68} and \eqref{md69} yields 
  \begin{equation}\label{md67}   
    0<\Phi_{\mathcal{C}}(t) \implies 
    t<s_{\mathcal{C}}=\dim_H \Lambda_{\mathcal{C}}\leq \dim_H \Lambda,
  \end{equation}
  and it holds for any $t\in (0,s_{\mathcal{F}})$. 
  Thus $s_{\mathcal{F}}\leq \dim_{\rm H} \Lambda$.

  When $s_{\mathcal{F}}> 1$, we can find a $\mathcal{C}\subset \mathcal{D}$ for which $\dim_H \Lambda_{\mathcal{C}}=1$. It is a simple consequence of Lemma \ref{md84}, Proposition \ref{md81} and \eqref{md68}. 
  Therefore the lower bound that covers both cases is 
  \begin{equation}\label{md66}
      \min\{s_{\mathcal{F}},1\}\leq \dim_{\rm H}\Lambda.
  \end{equation}

\end{proof}

\section{What makes a CPLIFS \nice{}?}\label{md24}

It is hard to check whether a CPLIFS $\mathcal{F}=\{f_k\}_{k=1}^m$ is \nice{}, that is if it satisfies definition \ref{md82}. In this section, we show by a case analysis that the following proposition holds.
\begin{proposition}\label{md47}
    Let $\mathcal{F}$ be a CPLIFS with generated self-similar system $\mathcal{S}$. If $\mathcal{S}$ satisfies the ESC, then $\mathcal{F}$ is \nice. 
  \end{proposition}
Theorem \ref{md46} is a straightforward consequence of Proposition \ref{md47} and Theorem \ref{md72}.

According to \cite[Corollary~1]{hofbauer1986piecewise}, if all functions in $\mathcal{F}$ are injective and the first cylinders are not overlapping, then $\mathcal{F}$ is \nice{}. 
This observation was utilized by Raith in the proof of \cite[Lemma~6]{raith1989hausdorff}.

In this section we always assume that $s\in(0,\dim_{\rm H}\Lambda]$.
The overlapping structures may induce multiple edges in the Markov diagram. In the associated matrix $\mathbf{F}(s)$ each multiple edge is represented as an entry of the form $\rho_{k_1,j_1}^s+\dots+\rho_{k_n,j_n}^s$ for some $n>1$. Since these entries can be bigger than $1$ in absolute value, the assumptions of \cite[Corollary~1]{hofbauer1986piecewise} do not hold. We need to investigate under which conditions can \cite[Corollary~1/ii]{hofbauer1986piecewise} help us. 

\begin{lemma}[F. Hofbauer {\cite[Corollary~1/ii]{hofbauer1986piecewise}}]\label{md65}
  Let $\mathcal{F}=\{f_k\}_{k=1}^m$ be a CPLIFS with Markov diagram $(\mathcal{D},\to)$ and associated matrix $\mathbf{F}(s)$. If $\mathbf{F}(s)$ can be written in the form 
  \[
    \mathbf{F}(s)=\begin{bmatrix}
      P & Q \\
      R & S
    \end{bmatrix}
  \]
  such that $\varrho (\mathbf{F}(s))>\varrho (S)$, then $\mathcal{F}$ is \nice{}. Here $P,Q,R,S$ are appropriate dimensional block matrices.
\end{lemma}

For the convenience of the reader we also present Hofbauer's proof here. 

\begin{proof} We follow the proof of Corollary 1/ii right after Theorem 9 in \cite{hofbauer1986piecewise}.
  Let $\lambda:=\varrho (\mathbf{F}(s))$ and $I_d$ be the $d$ dimensional identity matrix. We write $d_P$ and $d_F$ for the dimensions of the square matrices $P$ and $\mathbf{F}(s)$ respectively. It follows that $d_S=d_F-d_P$. We remark that $d_F$ and $d_S$ may not be finite.

  As $\lambda>\varrho(S)$, $(I_{d_S}-xS)^{-1}=\sum_{k=0}^{\infty} x^kS^k$ exists for $\vert x\vert\leq\lambda^{-1}$ and has nonnegative entries for $0\leq x\leq \lambda^{-1}$. For $E(x)=P+xQ(I-xS)^{-1}R$ we have the following matrix equation
  \begin{equation}\label{md64}
    \begin{bmatrix}
      I-xE(x) & -xQ(I-xS)^{-1} \\
      0 & I
    \end{bmatrix} 
    \begin{bmatrix}
      I & 0 \\
      -xR & I-xS
    \end{bmatrix} =
    I-x\mathbf{F}(s),
  \end{equation}
  for all $\vert x\vert\leq \lambda^{-1}$. Since $\lambda=\varrho(\mathbf{F}(s))$, we find an $x$ with $|x|=\lambda^{-1}$ such that $I-x\mathbf{F}(s)$ is not invertible. Fix this $x$ number. By \eqref{md64}, knowing that both $I$ and $(I-xS)^{-1}$ are invertible, we get that $I-xE(x)$ is not invertible, i.e. $\varrho (E(x))\geq \lambda$. Since the entries of $E(|x|)$ are greater than or equal to the absolute values of the entries of $E(x)$, we get $\varrho (E(\lambda^{-1}))=\varrho (E(|x|))\geq \varrho (E(x))\geq \lambda$. Note that $E(x)$ is a finite matrix.
  
  For $t\in \left( 0,\lambda^{-1}\right]$ the map $t\to \varrho (E(t))$ is continuous and increasing, since the entries of $E(t)$ are continuous and increasing in $t$. Since $\varrho(E(\lambda^{-1}))\geq \lambda$, we find a $y\in \left( 0,\lambda^{-1}\right]$ with $\varrho (E(y))=y^{-1}$. Since $E(y)$ has nonnegative entries, this implies that $I-yE(y)$ is not invertible. Hence $I-y\mathbf{F}(s)$ is not invertible by \eqref{md64}. As $\lambda=\varrho (\mathbf{F}(s))$, we get $y=\lambda^{-1}$. Since $E(y)$ is a finite matrix, we find a nonnegative vector $u_1$ with $u_1(I-yE(y))=0$. Set $u_2=yu_1Q(I-yS)^{-1}$, which is a nonnegative $l^1(d_S)$ vector, as the rows of $Q$ are in $l^1(d_S)$. Hence $u=(u_1,u_2)$ is a nonnegative $l^1(d_F)$ vector and $u(I-y\mathbf{F}(s))=0$ by \eqref{md64}. That is $u$ is a left eigenvector for $\lambda=\varrho(\mathbf{F}(s))$.

  Similar calculation for the transpose of $\mathbf{F}(s)$ yield a nonnegative $l^{\infty}(d_F)$ vector $v$ with $(I-\lambda^{-1}\mathbf{F}(s))v=0$.
  
\end{proof}

Lemma \ref{md84} implies that for $s\in (0,\dim_{\rm H}\Lambda)$ we have $\varrho (\mathbf{F}(s))>1$, where $\Lambda$ is the attractor of the CPLIFS $\mathcal{F}$. Therefore, in order to apply Lemma \ref{md65}, it is enough to show that 
\begin{equation}\label{md61}
  \lim_{N\to\infty} \varrho \left( \mathbf{F}_{\mathcal{D}\setminus\mathcal{D}_N}(s) \right) =1.
\end{equation}
If \eqref{md61} holds, then $\mathbf{F}_{\mathcal{D}\setminus\mathcal{D}_N}(s)$ can take the place of the submatrix $S$ in Theorem \ref{md65} for a big enough $N$. 

In the special case of expansive piecewise monotonic mappings, \eqref{md61} was verified by F. Hofbauer \cite[Corollary~1/i]{hofbauer1986piecewise}.
To extend his results to CPLIFS, we need to show the same for our expansive multi-valued mappings $T$. The only difference between our and his Markov diagrams is the occurence of multiple edges, caused by the possible overlappings. We note that not all of the overlappings induce multiple edges, as monotonicity intervals of the same level might overlap.

\begin{definition}
    Let $Z\in \mathcal{Z}$ be an element of the base partition, and let $f_{k_1,j_1}^{-1}, f_{k_2,j_2}^{-1}$ be two different branches of the expansive multivalued mapping $T$. We say that $f_{k_1,j_1}^{-1}$ and $f_{k_2,j_2}^{-1}$ \textbf{cause an overlap on} $Z$ if 
    \[
        \mbox{int}(f_{k_1,j_1}^{-1}(Z))\cap \mbox{int}(f_{k_2,j_2}^{-1}(Z))\neq\emptyset, 
    \]
    where $\mbox{int}(A)$ denotes the inerior of the set $A$.
    If $\exists z\in Z: f_{k_1,j_1}^{-1}(z)=f_{k_2,j_2}^{-1}(z)$, then we call it a \textbf{cross overlap}, otherwise we call it a \textbf{light overlap}. See Figure \ref{md62}. We call the branches that cause an overlap over the same interval \textbf{cross overlapping branches} or \textbf{light overlapping branches}, respectively.
\end{definition}

\begin{figure}[t]
  \centering
  \includegraphics[width=7cm]{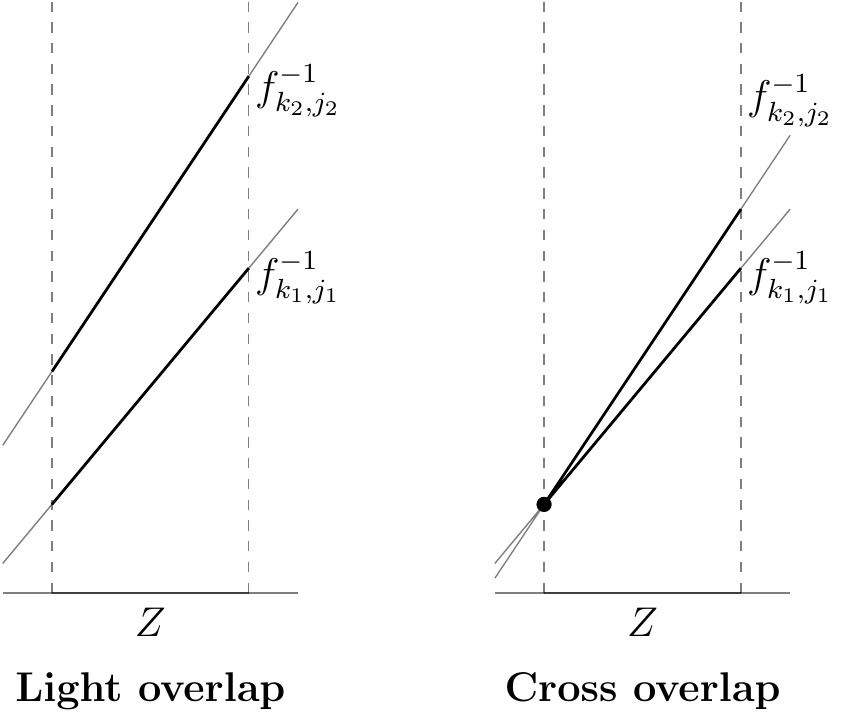}
  \caption{The two types of overlappings.}\label{md62}
\end{figure}

Note that the graphs of the branches of $T$ can only intersect at the endpoint of some base interval $Z\in\mathcal{Z}_0$ (see Definition \ref{md56}).
We say that the \textbf{\ordero} is $K$ if the maximal number of branches of $T$ that have intersecting domains is $K$.

\subsection{The case of light overlaps}

\begin{lemma}\label{md63}
  Let $\mathcal{F}$ be a CPLIFS with only light overlaps. Then, there exists a finite partition $\mathcal{Y}$ such that the Markov diagram of $\mathcal{F}$ with respect to $\mathcal{Y}$ do not contain any multiple edges.
\end{lemma}

\begin{proof}
  Let $K$ be the \ordero{} of $\mathcal{F}$ and $T$ be the multi-valued mapping associated to $\mathcal{F}$.
  First assume that the branches of $T$ overlap only above $Z\in\mathcal{Z}_0$ and write $f_{k_1,j_1}^{-1},\dots ,f_{k_K,j_K}^{-1}$ for these branches. Since we only have \lighto{}, without loss of generality we may assume that $\forall x\in Z: f_{k_{\beta},j_{\beta}}^{-1}(x)<f_{k_{\gamma},j_{\gamma}}^{-1}(x)$ if $\beta<\gamma$.

  Let us define  
  \begin{align}\label{md34}
      \varepsilon:=\max \big\{ \varepsilon^{\prime}&>0 \big\vert\: \forall \beta\in[K-1], \forall A\subset \mathbb{R}, |A|=\varepsilon^{\prime}: \\ 
      &f_{k_{\beta+1},j_{\beta+1}}^{-1}(A)\cap  f_{k_{\beta},j_{\beta}}^{-1}(A)=\emptyset\big\}. \nonumber
  \end{align}
  Since we only have light overlaps, $\varepsilon$ is a well-defined positive number.
  The images of any interval $A\subset Z$  with length at most $\varepsilon$ by the branches $f_{k_1,j_1}^{-1},\dots,$ $f_{k_K,j_K}^{-1}$ must be disjoint. It is illustrated on Figure \ref{md11}.
  \begin{figure}[b]
    \centering
    \includegraphics[width=4cm]{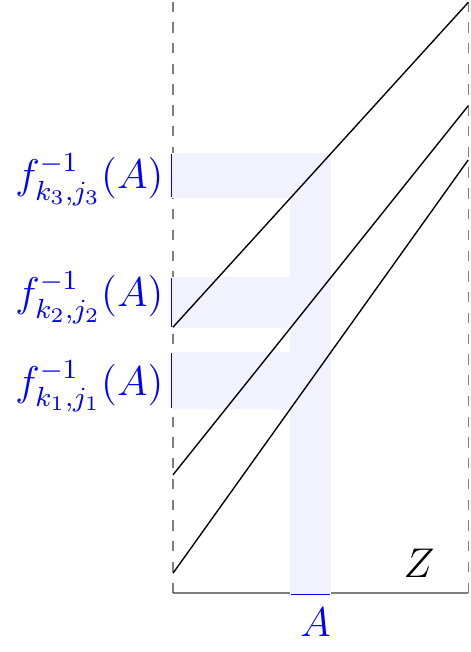}
    \caption{If the interval $A\subset Z$ is small enough, then its images by the three light overlapping branches are disjoint.}\label{md11}
  \end{figure}
  
  Let $\mathcal{Y}_Z$ be a finite partition of $Z$ whose elements are all have length at most $\varepsilon$. By substituting $\mathcal{Y}_Z$ in the place of $Z$ in $\mathcal{Z}_0$, we obtain a finite refinement $\mathcal{Y}$ of $\mathcal{Z}_0$. By \eqref{md34}, there are no multiple edges in the Markov diagram of $\mathcal{F}$ with respect to $\mathcal{Y}$.

  Assume now that light overlaps occur above $q>1$ many monotonicity intervals $Z_1,\dots Z_q\in \mathcal{Z}_0$. For each $i\in[q]$, let $\varepsilon_i$ be the number defined in \eqref{md34} using the branches above $Z_i$, and let $\mathcal{Y}_{Z_i}$ be a finite partition of $Z_i$ whose elements are all have length at most $\varepsilon_i$.
  By replacing $Z_i$ in $\mathcal{Z}_0$ with $\mathcal{Y}_{Z_i}$ for every $i\in[q]$, we obtain the finite partition $\mathcal{Y}$. The Markov diagram of $\mathcal{F}$ with respect to $\mathcal{Y}$ does not contain any multiple edges. 
\end{proof}

Lemma \ref{md63} implies that for a CPLIFS with only light overlaps \cite[Corollary~1/i]{hofbauer1986piecewise} also holds. 

\begin{proposition}\label{md42}
  Let $\mathcal{F}$ be a CPLIFS with only light overlaps, and for a finite partition $\mathcal{Y}$ let $(\mathcal{D}(\mathcal{Y}),\to)$ be the Markov diagram of $\mathcal{F}$ with respect to $\mathcal{Y}$. Then there exists a $\mathcal{Y}$ finite partition such that  
  \begin{equation}\label{md33}
      \lim_{N\to\infty} \varrho \left( \mathbf{F}_{\mathcal{D}(\mathcal{Y})\setminus\mathcal{D}(\mathcal{Y})_N}(s) \right) =1,
  \end{equation}
  where $\mathbf{F}(s)$ is the matrix associated to $(\mathcal{D}(\mathcal{Y}),\to)$.
\end{proposition}
For the convenience of the reader, we include here a modified version of the proof of \cite[Corollary~1/i]{hofbauer1986piecewise}. 

\begin{proof}
  According to Lemma \ref{md63}, there exists a $\mathcal{Y}$ finite refinement of $\mathcal{Z}_0$ such that there are no multiple edges in $(\mathcal{D},\to):=(\mathcal{D}(\mathcal{Y}),\to)$.
  Fix $N>1$, and let $Z\in \mathcal{D}\setminus \mathcal{D}_N$. Further, let $(k_1,j_1)$ be the label of one of the edges from $Z$, and let $(k_1,j_1),\dots ,(k_q,j_q)$ be a sequence of labels corresponding to a path of directed edges in $\mathcal{D}\setminus \mathcal{D}_N$ for some arbitrary $q>0$. We will show that if $q\leq N$, then $(k_1,j_1),\dots ,(k_q,j_q)$ defines at most two directed paths of the form $Z_0=Z\to Z_1\to \dots\to Z_q$ in $\mathcal{D}\setminus \mathcal{D}_N$.

  Assume that $q<N$ and one of the endpoints of $Z$ is a critical point.
  Without loss of generality suppose that $Z=[w,x]$ where $w\in \mathcal{K}$. The successors of $Z$ by the branch $(k_1,j_1)$ can only end in $f_{k_1,j_1}^{-1}(w), f_{k_1,j_1}^{-1}(x)$ or at some critical point. Out of them only at most one is in $\mathcal{D}\setminus \mathcal{D}_N$, since intervals of the form $[a,b]$ where $a\in T^{i_1}v_1, b\in T^{i_2}v_2, \: v_1,v_2\in\mathcal{K}, 0\leq i_1,i_2\leq N$ are all contained in $\mathcal{D}_N$. Namely, the interval which ends in $f_{k_1,j_1}^{-1}(x)$. Therefore $Z_1$ is uniquely defined.
  Similarly, $Z_i$ must be that sucessor of $Z_{i-1}$ which ends in $f_{k_i,j_i}^{-1}\circ\cdots\circ f_{k_1,j_1}^{-1}(x)$, for $i\in[q]$. So in this case $Z_1,\dots ,Z_q$ are uniquely defined.

  If none of the endpoints of $Z=[x,y]$ is a critical point, then there are at most two successors of $Z$ in $\mathcal{D}\setminus \mathcal{D}_N$. Both of these intervals end in a critical point, so we can apply the previous argument for them. Thus we have two versions for $Z_1,\dots ,Z_q$.

  We just showed that in the matrix $\left( \mathbf{F}_{\mathcal{D}\setminus \mathcal{D}_N}(s) \right)^{nN}$, in the row of an arbitrary $Z\in\mathcal{D}\setminus \mathcal{D}_N$ there are at most $2^n\cdot K^N$ many non-zero elements for all $n>0$. Here we used $K^N$ as an upper bound for the possible number of $N$ length paths in $\mathcal{D}\setminus \mathcal{D}_N$. 
  It follows from Lemma \ref{md63} that the elements of $\left( \mathbf{F}_{\mathcal{D}\setminus \mathcal{D}_N}(s) \right)^{nN}$ are all upper bounded by $1$, since there are no multiple edges in $(\mathcal{D},\to)$.
  Thus 
  \begin{equation}\label{md38}
    \varrho\left( \mathbf{F}_{\mathcal{D}\setminus \mathcal{D}_N}(s)\right) \leq 
    \sqrt[nN]{\Vert \left( \mathbf{F}_{\mathcal{D}\setminus \mathcal{D}_N}(s) \right)^{nN} \Vert_{\infty}} \leq 
    \sqrt[nN]{2^n K^N}=\sqrt[N]{2}\cdot \sqrt[n]{K},
  \end{equation}
  for any $n\geq 1$, and with this the statement is proved.
\end{proof}

Lemma \ref{md65} and Proposition \ref{md42} together gives 
\[
  \mathcal{F} \mbox{ has only \lighto } \implies \mathcal{F} \mbox{ is \nice.}
\]
That is for a CPLIFS $\mathcal{F}$ with only \lighto and with a generated self-similar system satisfying the ESC, we always have $\dim_{\rm H}\Lambda=\min\{1,s_{\mathcal{F}}\}$, where $\Lambda$ is the attractor of $\mathcal{F}$.

\subsection{The case of \crosso}

We call the elements of the set 
\[
    \left\{ x\in \mathcal{I} \big\vert \exists k_1,k_2\in [m], 
    \exists j_1\in[l(k_1)], \exists j_2\in[l(k_2)]:  
    f_{k_1, j_1}^{-1}(x)=f_{k_2, j_2}^{-1}(x) \right\}
\]
\textbf{intersecting points}. They form a subset of the critical points $\mathcal{K}$. Let $w\in I$ be an intersecting point, then the elements of $\mathcal{D}$ can only contain $w$ as their endpoint. If $D\in \mathcal{D}$ ends in $w$, then we say that $D$ \textbf{is causing cross overlaps} at $w$.

\begin{lemma}\label{md50}
  Let $\mathcal{F}$ be a CPLIFS with associated expanding multi-valued mapping $T$. 
  Let $x_0\in I$ be an intersecting point. 
  If the generated self-similar system $\mathcal{S}$ of $\mathcal{F}$ satisfies the ESC, then there is no finite $N$ for which $x_0\in T^{N}(x_0)$. 
\end{lemma}

\begin{proof}
  We will prove the statement by contradiction and assume that there is a finite $N>0$ such that $x_0\in T^{N}(x_0)$.  
  Let $f_{k_1^{\prime},j_1^{\prime}}^{-1}$ and $f_{\widehat{k}_1,\widehat{j_1}}^{-1}$ be two different branches of $T$ that maps $x_0$ to the same value. These must exist since $x_0$ is an intersecting point.
  Without loss of generality, assume that the sequence of branches $\left( (k_1,j_1),\dots ,(k_N,j_N) \right)$ maps $x_0$ to itself. Precisely,
  \[
    f_{k_N,j_N}^{-1} \circ\dots\circ f_{k_1,j_1}^{-1} (x_0) = x_0. 
  \]
  The same holds for the sequence of branches $\left( (k_1,j_1),\dots ,(k_N,j_N),(k_1^{\prime},j_1^{\prime}) \right)$ and $\left( (k_1,j_1),\dots ,(k_N,j_N),(\widehat{k}_1,\widehat{j_1}) \right)$ as well. 

  For a given branch $f_{k,j}^{-1}$, we write $S_{k,j}$ for the corresponding element of the generated self-similar IFS $\mathcal{S}$. It follows that 
  \begin{equation}\label{md31}
    S_{(k_1,j_1),\dots ,(k_N,j_N),(k_1^{\prime},j_1^{\prime})} \equiv 
    S_{(k_1,j_1),\dots ,(k_N,j_N),(\widehat{k}_1,\widehat{j_1})}.
  \end{equation} 
  Using these two functions, we can construct at least two identical iterates with different codes for any level $n>N$. 
  It implies that the ESC fails for $\mathcal{S}$.
\end{proof}

\begin{lemma}\label{md44}
  Let $\mathcal{F}$ be a CPLIFS whose generated self-similar system satisfies the ESC.
  Let $T$ be the expanding multi-valued mapping associated to $\mathcal{F}$ and $W$ be the set of all intersecting points. Fix $P>0$. Then there exists a finite refinement $\mathcal{Y}$ of $\mathcal{Z}_0$ such that 
  \begin{equation}\label{md43}
    \forall Z\in \mathcal{Y}, \forall w\in W, \forall n\in[P]:
    w\in Z\implies Z\cap \left( \cup T^n(Z) \right) =\emptyset.
  \end{equation} 
\end{lemma}

\begin{proof}
  Let $w\in W$ be an arbitrary intersecting point.
  According to Lemma \ref{md50}, $\forall n\in[P]: w\not\in T^n(w)$. That is, the distance of $w$ and the set $\cup_{n\in[P]} T^n(w)$ is positive. Let $d>0$ be this distance. Recall that $\rho_{\min}$ denotes the smallest contraction ratio in $\mathcal{F}$, hence $1/\rho_{\min}$ is the largest slope of $T$. 

  Let $p\in \mathbb{R}$ be a point that satisfies 
  \begin{equation}\label{md30}
    \vert w-p\vert < \frac{d\rho^P_{\min}}{1+\rho^P_{\min}}.
  \end{equation}
  Let $Z\in \mathcal{Z}_0$ be a monotonicity interval that contains $w$, and let $p$ be the only point in $Z$ that satisfies \eqref{md30}. Intersecting points are also critical points, thus $Z$ can only contain one such point. We cut $Z$ into two closed intervals by $p$ and call them $Y_Z,Y^{\prime}_Z$.

  We can construct the pair of intervals $Y_Z,Y^{\prime}_Z$ for any monotonicity interval $Z\in \mathcal{Z}_0$ that contains an intersecting point $w\in W$. 
  By replacing all $Z$ in $\mathcal{Z}_0$ that causes cross overlaps with the correspoding 
  $\{Y_Z,Y^{\prime}_Z\}$, we obtain a finite partition $\mathcal{Y}$ that satisfies \eqref{md43}.
\end{proof}

\begin{proof}[Proof of Proposition~\ref{md47}]
  We are going to construct a \nice{} partition of $\mathcal{F}$ with the help of Lemma \ref{md44} and Lemma \ref{md63}.
  Write $M$ for the number of intersection points in the system and $K$ for the \ordero.
  We know that $\varrho \left( \mathbf{F}(s) \right)>1$, so we can fix an $\varepsilon>0$ for which $\varrho \left( \mathbf{F}(s) \right)>1+\varepsilon$.
  
  Fix a $P>0$ big enough such that 
  \begin{equation}\label{md28}
      \sqrt[P]{K} < \sqrt[M+1]{1+\varepsilon}.
  \end{equation}
  We apply Lemma \ref{md44} to $\mathcal{F}$ and $P$ to obtain the finite partition $\mathcal{Y}$. Let $(\mathcal{D},\to)$ be the Markov diagram of $\mathcal{F}$ with respect to $\mathcal{Y}$, and let $Z\in \mathcal{Y}$ be an interval that causes a cross overlapping. Thanks to the construction of $\mathcal{Y}$, $Z$ does not intersect with its first $P$ successors. In other words, no $P$ length directed path in $\mathcal{D}\setminus \mathcal{D}_P$ can visit $Z$ more than once.

  Let $Z_{\rm cross}$ be the set of all images of $Z$ by the different cross overlapping branches defined above it. The elements of $Z_{\rm cross}$ are nested. Therefore, using the cross overlapping branch with the biggest expansion ratio, we can dominate every directed path of length at most $P$ in $\mathcal{D}\setminus \mathcal{D}_P$ that goes through $Z$ and contains an element of $Z_{\rm cross}$. This means that for every $n\in[P]$ and for every directed path $\widehat{Z}\to C_1\to\dots\to C_n$ in $\mathcal{D}\setminus \mathcal{D}_P$ with $\widehat{Z}\subset Z$ and $C_1\in Z_{\rm cross}$, there exists a directed path $Z\to D_1\to\dots\to D_n$ in $\mathcal{D}\setminus \mathcal{D}_P$ such that $D_1$ is the successor of $Z$ by a branch of the biggest slope, and $\forall k\in[n]: C_k\subset D_k$.
  It is essentially the same as erasing all other cross overlapping branches of $T$ above $Z$, see Figure \ref{md26}.
  We do the same domination for all $Z\in \mathcal{Y}$ that causes a cross overlapping.
  Let $\mathbf{F}^{\max}(s)$ be the matrix of this dominated system. We write $\max$ in the upper index to indicate that after the domination the only cross overlapping branch left above each $Z$ that originally caused a cross overlap is the one with the biggest expansion ratio.

  \begin{figure}[t]
    \centering
    \includegraphics[width=8cm]{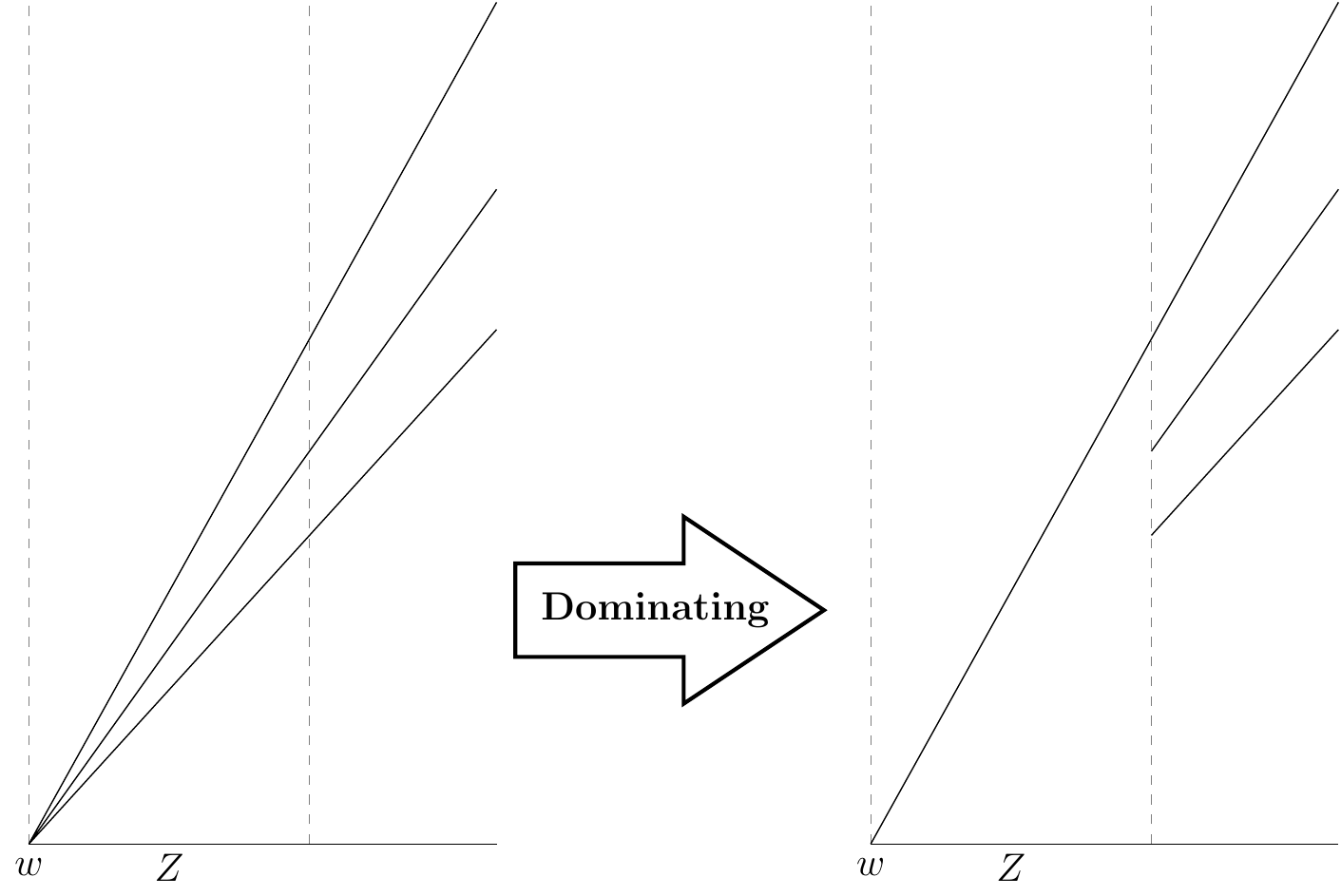}
    \caption{This figure illustrates how we handle cross overlappings by dominating the branches around the intersecting point $w$ with a branch of the biggest slope. The nodes of the Markov diagram remain the same, we only delete some edges.}\label{md26}
  \end{figure}
   
  Our new system, the one we obtained by dominating the cross overlapping branches, can only have light overlaps. The Markov diagram of this system has the same nodes as the original, we only erased edges by the domination. Using Proposition \ref{md42}, we get the finite refinement $\mathcal{Y}^{\prime}$ of $\mathcal{Y}$ for which \eqref{md33} holds with $\mathbf{F}^{\max}(s)$. 
  That is there exists an $N>P$ such that 
  \begin{equation}\label{md29}
    \varrho \left( \mathbf{F}^{\max}_{\mathcal{D}^{\prime}\setminus\mathcal{D}^{\prime}_N}(s) \right) < 1+\varepsilon,
  \end{equation}
  where $(\mathcal{D}^{\prime},\to)$ is the Markov diagram of $\mathcal{F}$ with respect to $\mathcal{Y}^{\prime}$.
  Let $\mathbf{F}^{\prime}(s)$ be the matrix associated to $(\mathcal{D}^{\prime},\to)$. It follows from the construction of the matrix $\mathbf{F}^{\max}(s)$ that every entry of $\mathbf{F}^{\max}(s)$ is smaller or equal to the corresponding entry of $\mathbf{F}^{\prime}(s)$.

  Now we show that the submatrix 
  $\mathbf{F}^{\prime}_{\mathcal{D}^{\prime}\setminus\mathcal{D}^{\prime}_N}(s)$ has spectral radius smaller than that of $\mathbf{F}^{\prime}(s)$. 
  Let $Z\in \mathcal{Y}$ be one of those intervals that caused a cross overlapping before the domination. Observe that in $(\mathcal{D}^{\prime}\setminus\mathcal{D}^{\prime}_N,\to)$ at most $K$ many directed edges start from $Z$. That is we dominated at most $K^M$ many paths in $(\mathcal{D}^{\prime}\setminus\mathcal{D}^{\prime}_N,\to)$ with a single one. By this we obtain the upper bound  
  \begin{equation}\label{md48}
    \left( \Vert \left( \mathbf{F}^{\prime}_{\mathcal{D}^{\prime}\setminus\mathcal{D}^{\prime}_N}(s)\right)^{nP} \Vert_{\infty} \right)^{\frac{1}{nP}} \leq 
    \left(\sqrt[P]{K}\right)^M \left( \Vert \left( \mathbf{F}^{\max}_{\mathcal{D}^{\prime}\setminus\mathcal{D}^{\prime}_N}(s)\right)^{nP} \Vert_{\infty} \right)^{\frac{1}{nP}}, 
  \end{equation}
  for any $1\leq n$. 
  It follows that
  \begin{equation}\label{md41}
    \varrho \left( \mathbf{F}^{\prime}_{\mathcal{D}^{\prime}\setminus\mathcal{D}^{\prime}_N}(s)\right)
    \leq \left(\sqrt[P]{K}\right)^M\cdot 
    \varrho \left( \mathbf{F}^{\max}_{\mathcal{D}^{\prime}\setminus\mathcal{D}^{\prime}_N}\right).
  \end{equation}
  We conclude the proof by substituting \eqref{md29} and \eqref{md28} into \eqref{md41}
  \begin{equation}\label{md27}
    \varrho \left( \mathbf{F}^{\prime}_{\mathcal{D}^{\prime}\setminus\mathcal{D}^{\prime}_N}(s)\right)
    < 1+\varepsilon < \varrho \left( \mathbf{F}^{\prime}(s)\right).
  \end{equation}
  According to Theorem \ref{md65}, $\mathcal{F}$ is \nice{}.
\end{proof}

\bibliographystyle{abbrv}
\bibliography{CPLIFS_bib}

\end{document}